\PassOptionsToPackage{dvipsnames,table}{xcolor}
\documentclass[11pt,reqno,twoside]{amsart}
\usepackage[centering]{geometry}

\usepackage{mathtools,amsthm,amssymb,amsfonts}
\usepackage{graphicx}
\usepackage{verbatim}
\usepackage{geometry}
\usepackage{color}
\usepackage{tikz-cd}
\usepackage{array}   
\newcolumntype{C}{>{$}c<{$}}
\usetikzlibrary{cd,decorations.pathmorphing}
\usepackage{setspace}
\usepackage{todonotes}
\usepackage[utf8]{inputenc}
\usepackage{csquotes}
\usepackage[shortlabels]{enumitem}
\usepackage{float}
\usepackage{hyperref}
\usepackage[capitalise]{cleveref}

\usepackage{booktabs}
\usepackage{todonotes}

\usepackage{charter,eulervm}
\usepackage[mathscr]{euscript}

\theoremstyle{plain}
\newtheorem{theorem}{Theorem}[section]
\newtheorem*{theorem*}{The Contraction Principle}
\newtheorem{corollary}[theorem]{Corollary}
\newtheorem{proposition}[theorem]{Proposition}

\theoremstyle{definition}
\newtheorem*{definition*}{Definition}

\theoremstyle{remark}
\newtheorem{remark}[theorem]{Remark}
\newtheorem*{notation*}{Notation}

\let\Lpolish\L
\def\L{\mathcal{L}}

\newcommand{\N}{\mathbb{N}}

\newcommand{\Rn}{\mathbb{R}}

\newcommand{\ignore}[1]{}

\newcommand{\st}{\mid}

\providecommand\qc{compact}%
\providecommand\qcop{\mathop {\mathring{{\mathcal{K}}}}\nolimits}%

\providecommand\Dm{{\mathfrak{m}}}
\providecommand\Dp{{\mathfrak{p}}}

\providecommand\lra{\longrightarrow}
\providecommand\inv{{\mathrm{inv}}}
\providecommand\Spec{{\mathrm{Spec\,}}}

\providecommand\CO{{\mathcal{O}}}
\providecommand\cp[1]{{\mathrm{cp}(#1)}}
\newcommand\jsub{join-subfit}\newcommand\msub{meet-subfit}

\title{Subfitness in distributive (semi)lattices}
\author{G.~Bezhanishvili}\address{guram@nmsu.edu, New Mexico State University, NM, USA}
\author{J.~Madden}\address{jamesjmadden@gmail.com, Louisiana State University Baton Rouge, LA, USA}
\author{M.~A.~Moshier}\address{moshier@chapman.edu, Chapman University, CA, USA}
\author{M.~Tressl}\address{marcus.tressl@manchester.ac.uk, The University of Manchester, U.K.}
\author{J.~Walters-Wayland}\address{walterswayland@chapman.edu, Chapman University, CA, USA}

\date{}

\begin{document}


\begin{abstract}
    We investigate
    whether the set of subfit elements of a distributive semilattice is an ideal. This question was raised by the second author at the BLAST conference in 2022.
    We show that in general it has a negative solution, however if the semilattice is a lattice, then the solution is positive.
    This is somewhat unexpected since, as we show, a semilattice is subfit if and only if so is its distributive lattice envelope.
\end{abstract}

\maketitle

\tableofcontents

\providecommand{\up}{\mathord{\uparrow}}


\section{Introduction}

\noindent
The property of subfitness was first studied by Wallman in \cite{wallman_lattices_1938} under the name ``disjunction property.''  According to the definition in his paper, a meet-semilattice $L$  with least element zero has this property if: ``[whenever] $a$ and $b$ are different  elements of $L$ there is an element $c$ of $L$ such that one of $a\wedge c$ and $b\wedge c$ is zero and the other is not zero'' (page 115).
Evidently, the property can be defined without reference to the order.  Let $(G, \cdot)$ be a commutative semigroup  with \textbf{absorbing element} $z$ (i.e., $x\cdot z = z$ for all $x\in G$, see \cite[1.6]{golan_semirings_1999}).
We say $G$ is \textbf{subfit} if, whenever $a$ and $b$ are different  elements of $G$, there is an element $c\in G$ such that one of $a\cdot c$ and $b\cdot c$ is equal to $z$ and the other is not; in other words,
\[
\{\,c\in G\mid a\cdot c =z\,\}=\{\,c\in G\mid b\cdot c =z\,\}\implies a=b.
\]
Based on this interpretation, the property was studied in a purely algebraic setting by Pierce in \cite{pierce_homomorphisms_1954}.
Later, in the 1970s, subfitness was recognized as an important separation property in pointfree topology; see \cite{isbell_atomless_1972, simmons_lattice_1978, picado_separation_2021}.  A frame is subfit if its underlying \textit{join}-semilattice has this property, i.e., if for any different elements $a,b$ there is some $c$ such that one of $a\vee c$ and $b\vee c$ is equal to $1$ and the other is not.

In the present paper, we solve a problem stated by the second coauthor in 2022:  ``In a join-semilattice $A$, if the principal ideals ${\downarrow} a$ and ${\downarrow} b$ are subfit, does it follow that ${\downarrow}(a\vee b)$ is subfit?''  (Note that $x\in A$ is an absorbing element in  ${\downarrow} x\subseteq A$.)  We show that this is true if $A$ is the underlying join-semilattice of a distributive lattice.  It is not enough, however, to assume merely that $A$ is a distributive join-semilattice, and we provide an example.  We shall give more details about the problem shortly, but before we do, we need to clear up some points of terminology suggested by the first paragraph.

It is reasonable to view subfitness as a property of semilattices, where we understand a semilattice to be a commutative, idempotent monoid $(A,\cdot,e)$.  In practice, semilattices are often treated as posets, but there are two ways to do this.  The operation $\cdot$  may be interpreted as join $\vee$, with the identity being the bottom element and the absorbing element, if there is one, the top.
Or the operation $\cdot$ may be interpreted as meet $\wedge$, with the identity being the top element and the absorbing element, if there is one, the bottom.  Semilattices often occur within other structures.
For example, in a bounded lattice  $L=(L,\vee, \wedge, 0, 1)$, the element $0$ is the absorbing element for $\wedge$ and the identity for $\vee$, while $1$ is  the absorbing element for $\vee$ and the identity for $\wedge$.  It is possible for $L$ to be subfit with respect to either operation, or both, or neither.
For this reason, it is useful to introduce the terms ``join-subfit'' and ``meet-subfit,'' and indeed this terminology plays an important role in the  present paper.  Our main theorem is stated and proved most naturally in the language of join-semilattices, because it is inspired by topology.  On the other hand, our main example is naturally described in the language of meet-semilattices, because it is more natural to think about meet-semilattices of finite sets than about join-semilattices of co-finite sets.

Let $A=(A,\cdot, e)$ be a semilattice.
For $a\in A$, consider the set $a^\dagger:=\{\,x\in A\mid x\cdot a = a\,\}$.
Then $a^\dagger$ is a subsemilattice of $A$ with absorbing element $a$.
We say that $a$ is a \textbf{subfit element} of $A$ if $a^\dagger$ is subfit.
If $A$ is viewed as a join-semilattice, then $a^\dagger ={\downarrow} a$ and we say
$a$ is
\textbf{\jsub}; similarly if $A$ is viewed as a meet-semilattice, then $a^\dagger ={\uparrow}a$ and we say $a$ is \textbf{\msub}.

A topology, seen as a join-semilattice, is join-subfit iff every nonempty locally closed set contains a nonempty closed set
(cf.~\cite[p.~23]{picado_separation_2021}). This shows that join-subfitness of a topology is near the  T$_1$ separation property. In fact, a topology is $T_1$ iff it is $T_D$ and join-subfit; see \cite[p.~24]{picado_separation_2021}. Subfitness also has a prominent role in algebraic geometry: If $X$ is the prime spectrum of a commutative ring $R$, then its frame of radical ideals is join-subfit if and only if  $R$ is a Jacobson ring, i.e.~every radical ideal of $R$ is the intersection of maximal ideals of $R$.
Prominent examples are finitely generated algebras over a given Jacobson ring;
every homomorphism $\phi:R\lra S$ of such algebras induces a map between the maximal ideal spectra $\mathrm{Max}(S)\lra \mathrm{Max}(R),\Dm\mapsto \phi^{-1}(\Dm)$ (see \cite[Thm.~12.3.12 and Rem. 12.3.14]{dickmann_spectral_2019}). The significance of join-subfitness for ring theory in connection with Jacobson rings was pointed out by Simmons \cite{simmons_lattice_1978}  where the property was named \textit{conjunctive};
see also \cite{delzell_conjunctive_2021}.

A join-semilattice $A$ is said to be \textbf{distributive} if for all $a,b,c\in A$ with $c\leq a\lor b$, there are $a', b'\in A$ such that $a'\leq a$, $b'\leq b$, and $c = a'\lor b'$.

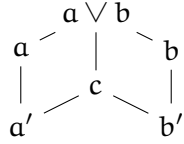
\begin{center}
\begin{tikzpicture}[scale=0.50]
  \node (a) at (-2,2) {$a$};
  \node (b) at (2,2) {$b$};
  \node (d) at (-2,0) {$a^\prime$};
  \node (c) at (0,1) {$c$};
  \node (e) at (2,0) {$b^\prime$};
  \node (f) at (0,3) {$a \vee b$};
  \draw (a) -- (d) -- (c)  -- (e) -- (b) ;
  \draw (c) -- (f) ;
  \draw[preaction={draw=white, -,line width=6pt}] (a) -- (f) -- (b);
\end{tikzpicture}
\end{center}
The version of the property for meet-semilattices is analogous. Note that a lattice is distributive as a join-semilattice iff  it is
distributive as a meet-semilattice iff it is a distributive lattice.\footnote{Here is the formulation of distributivity without reference to order.  We say \textit{$x$ absorbs $y$} if $x\cdot y = x$.  Then $(A,\cdot, e)$ is distributive if whenever $a\cdot b$ absorbs $c$, there are $a'$ and $b'$ that are absorbed respectively by $a$ and $b$ such that $a' \cdot b' = c$.}

At the BLAST conference at Chapman University in August 2022, coauthor Madden asked:
\textit{Do the \jsub\ elements of a distributive join-semilattice form an ideal?}  It is easy to show that in a distributive join-semilattice, if $a\leq b$ and $b$ is \jsub, then so is $a$.
So the problem is to show  that if  $a$ and $b$ are \jsub, then so is $a\vee b$; in other words, to show that the \jsub\ elements of $A$ are directed.   A simple example shows that the distributive hypothesis is necessary; in the lattice below, $a$ and $b$ are \jsub, but  $a\vee b$ is not.\looseness=-1

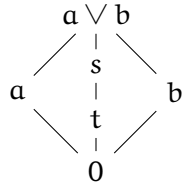
 \begin{center}
\begin{tikzpicture}[scale=.35]
  \node (one) at (0,6) {$a\vee b$};
  \node (s) at (0,4) {$s$};
  \node (t) at (0,2) {$t$};
  \node (zero) at (0,0) {$0$};
  \node (a) at (-3,3) {$a$};
  \node (b) at (3,3) {$b$};
  \draw (zero) -- (a) -- (one) -- (s) -- (t) -- (zero) -- (b) -- (one);
\end{tikzpicture}
\end{center}

\noindent
At the time it was asked,  the answer was known to be ``Yes'' for finite semilattices.  The question is motivated by the following considerations.  If a topological space is the union of two subspaces, both open and regular, the space may fail to be regular -- the line with a point doubled is the classical example, see \cite[Chap.~4, Ex.~5, p.~227]{munkres_topology_2000}. 
It is meaningful to ask if similar examples arise if we take subfitness in place of regularity.  The question arises in several more general contexts related to representations of algebras (see \cite{johnson_on_a_representation_1962}).  A related question was asked in \cite{delzell_conjunctive_2021}:
if $A$ is a join-semilattice such that ${\downarrow}a$ is \jsub\ for each $a\in A$, is $A$ ideally subfit?\footnote{The property of ideal subfitness was introduced in \cite{martinez_yosida_2006} and studied in \cite{delzell_conjunctive_2021}.
A join-semilattice $S$ is said to be \textit{ideally subfit} provided for all $u, v\in A$, if $u\nleq v$, then there is an ideal $W$ of $A$ such that, in the lattice of ideals of $A$, we have ${\downarrow}u\vee W=A$ and ${\downarrow}v\vee W\neq A$.}  Considering the relationship to the regularity problem, it is reasonable to ask whether the directedness condition is necessary, and this is the immediate motivation for the question.

In \cref{secLattice} we show that the question has a positive solution provided $A$ is the underlying join-semilattice of a distributive lattice,
and hence it has a positive solution in the realm of frames.
On the other hand, in \cref{DSLatCounterExample} we provide an example showing that for distributive semilattices
the answer is in general negative.
This is somewhat unexpected since we also show in \cref{secEnvelope} that an arbitrary semilattice is subfit if and only if its distributive lattice envelope is subfit.
Finally, in \cref{secTopology}, we use the representation theory for distributive (semi)lattices to provide a topological perspective on our results.

\section{Subfit elements in distributive lattices are directed} \label{secLattice}

\noindent
For a distributive join-semilattice $A$, let $S$ be the set of \jsub\ elements of $A$. It is straightforward to see that $S$ is a downset of $A$. In this paper we are concerned with the question of whether $S$ is an ideal. In this section we show that this is indeed the case provided $A$ is a lattice. Our proof below is constructive. In \cref{secTopology} we provide a topological explanation of this result, which does use the Prime Ideal Theorem.

\begin{theorem}\label{SubfitIsJoinStableInDLat}
    If $A$ is a bounded distributive lattice, then the set $S$ of \jsub\  elements of $A$ is an ideal of $A$. 
\end{theorem}

\begin{proof}
Let $a,b\in S$. To see that $a\vee b\in S$, it is sufficient to assume that $a\vee b=1$ and prove that $A$ is \jsub.
Suppose $s,t\in A$ with $t\nleq s$. We seek an element $z\in A$ such that $s\vee z <1$ and $t\vee z= 1$.
 If $a\wedge t\le a\wedge s$ and $b\wedge t\le b\wedge s$, then by distributivity, $(a\vee b)\wedge t\le (a\vee b)\wedge s$, so $t\le s$, a contradiction. Therefore,
we may assume without loss of generality that $b\wedge t\not\le b\wedge s$. Since ${\downarrow}b$ is \jsub, there is an element $y\leq b$ such that
\[
(b\wedge s)\vee y < b \mbox{ and } (b\wedge t)\vee y = b.
\]
The latter, by distributivity, implies that $b\le t\vee y$. Moreover, if $b\le s\vee y$, then
\[
(b\wedge s)\vee y=(b\vee y)\wedge (s\vee y)\ge b,
\]
contradicting the former. Thus, we have
\begin{equation}\label{eqy}
b\nleq s\vee y\;\text{ and }\;b\leq t\vee y.
\end{equation}
If $s\vee y\vee a<1$, let $z=y\vee a$ and we are done. On the other hand, if  $s\vee y\vee a = 1$, then
\begin{equation}\label{eqwnew}
(b\wedge a) \vee (b\wedge s)\vee y = (b\wedge (a\vee s)) \vee y = (b\vee y) \wedge (a\vee s\vee y) = b.
\end{equation}
Hence, setting $w=(b\wedge s)\vee y$ and using (\ref{eqwnew}), we have:
\begin{equation}\label{eqww}
(b\wedge s)\vee w = (b\wedge s)\vee y < b \text{ and } (a\wedge b)\vee w = (a\wedge b) \vee (b\wedge s) \vee y = b.
\end{equation}
From the first part of (\ref{eqww}), $b\nleq s\vee w$. From the second part of (\ref{eqww}), $b\leq a\vee w$.
By (\ref{eqy}),
\begin{equation}\label{eqwnew1}
t\vee w = t \vee (b\wedge s) \vee y = (t\vee b\vee y) \wedge (t\vee s\vee y) \ge b.
\end{equation}
Therefore, $b\leq (a\vee w) \wedge (t\vee w) = (a\wedge t)\vee w$.
Thus, $(a\wedge t)\vee w \nleq s\vee w$, so $a\wedge t \nleq s\vee w$, and hence $a\wedge t\nleq a \wedge (s\vee w)$. By \jsub ness of ${\downarrow}a$, there is  $x\leq a$ such that $(a\wedge t)\vee x=a$ and $(a\wedge (s\vee w))\vee x <a$,  i.e.,   $(a\wedge s) \vee (a\wedge w) \vee  x <a$.
If $a\le s\vee w\vee x$, then
\[
a\le (a\wedge s) \vee (a\wedge w) \vee (a\wedge x) = (a\wedge s) \vee (a\wedge w) \vee x,
\]
a contradiction. So $a \not\le s \vee w \vee  x$, and hence
$s\vee w \vee x <1$.  Also, $b\leq t\vee w$ by (\ref{eqwnew1}) and $a\leq t\vee x$ (because $a=(a\wedge t)\vee x$), so $t\vee w\vee x \ge a \vee b = 1$.
Let $z=w\vee x$, and we are done.
\end{proof}

\begin{remark}
To explain the logic of the proof, suppose $a, b, s, t\in A$ with $t\nleq s$.  If there is $x\leq a$ such that
$(a\wedge s)\vee x < a$ and $(a\wedge t)\vee x = a$, and there is $y\leq b$ such that
$(b\wedge s)\vee y < b$ and $(b\wedge t)\vee y = b$, then clearly, $t\vee x\vee y = a\vee b$, but it may also be that  $s\vee x\vee y = a\vee b$ (rather than $s\vee x\vee y < a\vee b$, as we desire).  This possibility accounts for some of the complexity in the proof above.

\end{remark}

\begin{remark}
The proof of \cref{SubfitIsJoinStableInDLat} is first-order and choice-free. In \cref{secTopology} we will give an alternative
proof using Stone duality.
\end{remark}

\noindent
Since each frame is a bounded distributive lattice, as an immediate consequence of  \cref{SubfitIsJoinStableInDLat} we obtain:

\begin{corollary}
    The set of \jsub\ elements of each frame is an ideal.
\end{corollary}

\section{Subfit elements in distributive semilattices are not directed} \label{DSLatCounterExample}

\noindent
We now give an example of a distributive join-semilattice $A$ for which the set of \jsub\ elements is not an ideal.
We present the example in the order-dual setup of meet-semilattices as this is closer to the intuition.
We thus give an example of a distributive meet-semilattice whose set of \msub\ elements is not a filter;
in fact, in this example there are elements $a$ and $b$ such that $a \wedge b =0$, both ${\uparrow}a$ and ${\uparrow}b$ are \msub, and yet $A$ is not \msub.

We construct $A$ as a meet-subsemilattice of the powerset $\mathcal P(\N)$ of $\N$.
Let $\mathcal F$ denote the set of finite subsets of $\N_{\geq 3}$ and $\mathcal C_{012}$ the set of cofinite subsets of $\N$ that contain $\{0,1,2\}$.  Let
\begin{itemize}
\item[] $\mathcal F_0:=\{\,F\cup \{0\}\mid F\in \mathcal F\,\}$
\item[] $\mathcal F_1:=\{\,F\cup \{1\}\mid F\in \mathcal F\,\}$
\item[] $\mathcal F_{01}:=\{\,F\cup \{0,1\}\mid F\in \mathcal F\,\}$
\item[] $\mathcal F_{12}:=\{\,F\cup \{1,2\}\mid F\in \mathcal F\,\}$
\end{itemize}
and let
\[
A := \mathcal F \cup \mathcal F_0 \cup \mathcal F_1  \cup \mathcal F_{01}  \cup \mathcal F_{12}  \cup \mathcal C_{012}.
\]
An illustration is given in \cref{fig:DA of fin/cofin}, where the cup shaped bubbles represent sets of finite sets and the unique umbrella shape at the top is a set of cofinite sets.
We set $a:= \{0\},\, b:=\{1\},\, c:=\{1,2\}$, and $\mathcal F_{012}:=\{\,F\cup \{0,1,2\}\mid F\in \mathcal F\,\}$.
Then $a,b,c\in A$ and $A$ is depicted in the black part of the figure. The set $\mathcal F_{012}$ is disjoint from $A$, but is contained in the sublattice of $\mathcal P(\N)$
generated by $A$.
\begin{figure}[!h]
\begin{center}
\hspace*{-0.7cm}
\begin{tikzpicture}[scale=0.334]

\begin{scope}[shift={(-1,-10)}]
    \draw (-2,0) arc [radius=2, start angle=180, end angle=360];
    \begin{scope}[yscale=0.7]
           \draw (-2,0) to [out=45, in=180] (-1.6,0.25);
            \draw (-1.6,0.25) to [out=0, in=135] (-1.2,0);
            \draw (-1.2,0) to [out=315, in=180] (-0.8,-0.25);
            \draw (-0.8,-0.25) to [out=0, in=45] (-0.4,0);
            \draw (-0.4,0) to [out=45, in=180] (0,0.25);
            \draw (0,0.25) to [out=0, in=135] (0.4,0);
            \draw (0.4,0) to [out=315, in=180] (0.8,-0.25);
            \draw (0.8,-0.25) to [out=0, in=45] (1.2,0);
            \draw (1.2,0) to [out=45, in=180] (1.6,0.25);
            \draw (1.6,0.25) to [out=0, in=135] (2,0);
        \end{scope}
\draw [fill=black] (0,-2) circle (2.0pt);
\node [below] at (0,-2) {$\scriptstyle \varnothing$};
\node at (0,-1) {$\scriptstyle \mathcal F$};
\end{scope}
\begin{scope}[shift={(-6,-5)}]
    \draw (-2,0) arc [radius=2, start angle=180, end angle=360];
    \begin{scope}[yscale=0.7]
           \draw (-2,0) to [out=45, in=180] (-1.6,0.25);
            \draw (-1.6,0.25) to [out=0, in=135] (-1.2,0);
            \draw (-1.2,0) to [out=315, in=180] (-0.8,-0.25);
            \draw (-0.8,-0.25) to [out=0, in=45] (-0.4,0);
            \draw (-0.4,0) to [out=45, in=180] (0,0.25);
            \draw (0,0.25) to [out=0, in=135] (0.4,0);
            \draw (0.4,0) to [out=315, in=180] (0.8,-0.25);
            \draw (0.8,-0.25) to [out=0, in=45] (1.2,0);
            \draw (1.2,0) to [out=45, in=180] (1.6,0.25);
            \draw (1.6,0.25) to [out=0, in=135] (2,0);
        \end{scope}
 \draw [fill=black] (0,-2) circle (2.0pt);
\node [below] at (0,-2) {$\scriptstyle a$};
\node at (0,-1) {$\scriptstyle \mathcal F_0$};
\end{scope}
\begin{scope}[shift={(4,-5)}]
    \draw (-2,0) arc [radius=2, start angle=180, end angle=360];
    \begin{scope}[yscale=0.7]
           \draw (-2,0) to [out=45, in=180] (-1.6,0.25);
            \draw (-1.6,0.25) to [out=0, in=135] (-1.2,0);
            \draw (-1.2,0) to [out=315, in=180] (-0.8,-0.25);
            \draw (-0.8,-0.25) to [out=0, in=45] (-0.4,0);
            \draw (-0.4,0) to [out=45, in=180] (0,0.25);
            \draw (0,0.25) to [out=0, in=135] (0.4,0);
            \draw (0.4,0) to [out=315, in=180] (0.8,-0.25);
            \draw (0.8,-0.25) to [out=0, in=45] (1.2,0);
            \draw (1.2,0) to [out=45, in=180] (1.6,0.25);
            \draw (1.6,0.25) to [out=0, in=135] (2,0);
        \end{scope}
\draw [fill=black] (0,-2) circle (2.0pt);
\node [below] at (0,-2) {$\scriptstyle b$};
\node at (0,-1) {$\scriptstyle \mathcal F_1$};
\end{scope}
\begin{scope}[shift={(-1,0)}]

    \draw (-2,0) arc [radius=2, start angle=180, end angle=360];
    \begin{scope}[yscale=0.7]
           \draw (-2,0) to [out=45, in=180] (-1.6,0.25);
            \draw (-1.6,0.25) to [out=0, in=135] (-1.2,0);
            \draw (-1.2,0) to [out=315, in=180] (-0.8,-0.25);
            \draw (-0.8,-0.25) to [out=0, in=45] (-0.4,0);
            \draw (-0.4,0) to [out=45, in=180] (0,0.25);
            \draw (0,0.25) to [out=0, in=135] (0.4,0);
            \draw (0.4,0) to [out=315, in=180] (0.8,-0.25);
            \draw (0.8,-0.25) to [out=0, in=45] (1.2,0);
            \draw (1.2,0) to [out=45, in=180] (1.6,0.25);
            \draw (1.6,0.25) to [out=0, in=135] (2,0);

        \end{scope}
\draw [fill=black] (0,-2) circle (2.0pt);
\node [below] at (0,-2.25) {$\scriptstyle a \vee b$};
\node at (0,-1) {$\scriptstyle \mathcal F_{01}$};
\end{scope}

\begin{scope}[shift={(9,0)}]
    \draw (-2,0) arc [radius=2, start angle=180, end angle=360];
    \begin{scope}[yscale=0.7]
           \draw (-2,0) to [out=45, in=180] (-1.6,0.25);
            \draw (-1.6,0.25) to [out=0, in=135] (-1.2,0);
            \draw (-1.2,0) to [out=315, in=180] (-0.8,-0.25);
            \draw (-0.8,-0.25) to [out=0, in=45] (-0.4,0);
            \draw (-0.4,0) to [out=45, in=180] (0,0.25);
            \draw (0,0.25) to [out=0, in=135] (0.4,0);
            \draw (0.4,0) to [out=315, in=180] (0.8,-0.25);
            \draw (0.8,-0.25) to [out=0, in=45] (1.2,0);
            \draw (1.2,0) to [out=45, in=180] (1.6,0.25);
            \draw (1.6,0.25) to [out=0, in=135] (2,0);
        \end{scope}
\draw [fill=black] (0,-2) circle (2.0pt);
\node [below] at (0,-2) {$\scriptstyle c$};
\node at (0,-1) {$\scriptstyle \mathcal F_{12}$};
\end{scope}

\begin{color}{red}
\begin{scope}[shift={(4,5)}]
    \draw (-2,0) arc [radius=2, start angle=180, end angle=360];
    \begin{scope}[yscale=0.7]
           \draw (-2,0) to [out=45, in=180] (-1.6,0.25);
            \draw (-1.6,0.25) to [out=0, in=135] (-1.2,0);
            \draw (-1.2,0) to [out=315, in=180] (-0.8,-0.25);
            \draw (-0.8,-0.25) to [out=0, in=45] (-0.4,0);
            \draw (-0.4,0) to [out=45, in=180] (0,0.25);
            \draw (0,0.25) to [out=0, in=135] (0.4,0);
            \draw (0.4,0) to [out=315, in=180] (0.8,-0.25);
            \draw (0.8,-0.25) to [out=0, in=45] (1.2,0);
            \draw (1.2,0) to [out=45, in=180] (1.6,0.25);
            \draw (1.6,0.25) to [out=0, in=135] (2,0);
        \end{scope}
        \draw [fill=red] (0,-2) circle (2.0pt);
\node [below] at (0,-2.25) {$\scriptstyle a \vee c$};
\node at (0,-1) {$\scriptstyle \mathcal F_{012}$};
\end{scope}
 \end{color}{red}
\begin{scope}[shift={(4,10)}]
    \begin{scope}[yscale=-1]
        \draw (-2,0) arc [radius=2, start angle=180, end angle=360];
        \begin{scope}[yscale=0.7]
               \draw (-2,0) to [out=45, in=180] (-1.6,0.25);
                \draw (-1.6,0.25) to [out=0, in=135] (-1.2,0);
                \draw (-1.2,0) to [out=315, in=180] (-0.8,-0.25);
                \draw (-0.8,-0.25) to [out=0, in=45] (-0.4,0);
                \draw (-0.4,0) to [out=45, in=180] (0,0.25);
                \draw (0,0.25) to [out=0, in=135] (0.4,0);
                \draw (0.4,0) to [out=315, in=180] (0.8,-0.25);
                \draw (0.8,-0.25) to [out=0, in=45] (1.2,0);
                \draw (1.2,0) to [out=45, in=180] (1.6,0.25);
                \draw (1.6,0.25) to [out=0, in=135] (2,0);
            \end{scope}
\end{scope}
\node at (0,1) {$\scriptstyle \mathcal C_{012}$};
\end{scope}

\draw [fill=black] (4,12) circle (2.0pt);
        \node [above] at (4,12) {$\scriptstyle \N$};

\draw[dashed] (2,5)--(2,10);
\draw[dashed] (6,5)--(6,10);

\draw (-6,-7)--(-1,-12)--(4,-7)--(9,-2);
\draw (-6,-7)--(-1,-2)--(4,-7);

\draw[dashed] (-8,-5)--(-3,0);
\draw[dashed] (2,-5)--(7,0);

\begin{color}{red}
\draw[dashed] (-3,0)--(2,5);
\draw[dashed] (11,0)--(6,5);
\draw (-1,-2)--(4,3)--(9,-2);
\end{color}{red}

\end{tikzpicture}
\caption{The meet-semilattice $A$ (black) within its distributive lattice envelope (black and red)}
         \label{fig:DA of fin/cofin}
         \end{center}
          \end{figure}
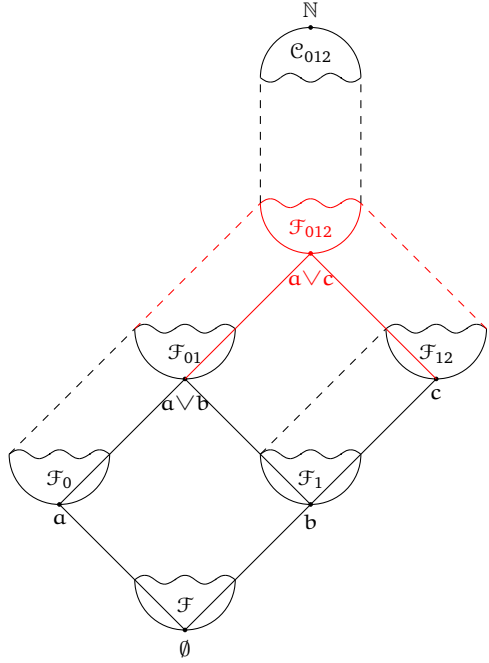
\noindent
The \textit{finite} sets in $A$ are those finite $E\subseteq \N$  such that
\[
E\cap\{0,1,2\}\in A_{012}:= \{\varnothing, \{0\}, \{1\}, \{0,1\}, \{1,2\}\}.
\]
Since $A_{012}$ is closed under intersection,  the intersection of any two finite elements of $A$ is an element of $A$.  It is clear that the intersection of a finite element with an infinite element, or the intersection of two infinite elements, is an element of $A$. Hence, $A$ is a bounded meet-subsemilattice of $\mathcal P(\N)$ under the operation of intersection.

Let $B$ be the sublattice of $\mathcal P(\N)$ generated by $A$. The set $B\setminus A$ is equal to $\mathcal F_{012}$ and is depicted in the diagram in red color.
In the language of \cref{secEnvelope}, $B$ is the distributive lattice envelope of $A$.

\noindent
Below, upsets are meant within the poset $A$.
Since $a= \{0\}\in A$ and $b=\{1\}\in A$,  we have
\[
{\uparrow} a =  \mathcal F_0  \cup \mathcal F_{01}  \cup \mathcal C_{012},\quad\text{and}\quad {\uparrow} b =  \mathcal F_{1} \cup \mathcal F_{01}  \cup \mathcal F_{12}  \cup \mathcal C_{012}.
\]
\noindent
{\it Claim 1}.\quad
    ${\uparrow} a$ is \msub.
\begin{proof}
    Suppose $x,y\in {\uparrow} a$ and $y\nsubseteq  x$.  We need to find $z\in {\uparrow}a$ with $x\cap z=a$ and $y\cap z\neq a$.
    Pick $n\in y\setminus x$. (Of course, $n\neq 0$.)  If $n\neq 2$, let $z=\{0,n\}$.  If $n=2$, then $y$ is cofinite and $x$ is finite. 
    Therefore, $y\setminus x$ contains an integer $m\geq 3$, and we set $z=\{0,m\}$.
\end{proof}
\noindent
{\it Claim 2}.\quad
    ${\uparrow} b$ is \msub.
\begin{proof}
    Suppose $x,y \in {\uparrow} b$ and $y\nsubseteq x$.  Pick $n\in y\setminus x$. Then $n\neq 1$, and $\{1,n\}\in {\uparrow} b$.  Set $z=\{1,n\}$.
\end{proof}
\noindent
{\it Claim 3}.\quad
    $A$ is not \msub.
\begin{proof}
    Consider $y=\{1,2\}$ (i.e.~$y=c$) and let $x$ be any element of $A$ that contains $1$ but not $2$.
    Then $y\nsubseteq x$, but the only non-zero elements of $A$ that are less than or equal to $y$ are $y$, $\{1\}$ and $\varnothing$, so the desired $z$ does not exist.
\end{proof}
\noindent
{\it Claim 4}.\quad
    ${\uparrow} a$ is a distributive lattice.
\begin{proof}
    In fact, ${\uparrow} a$ is a sublattice of $\mathcal P(\N)$, so it's a distributive lattice (with bottom element $a$ and top $\N$).
\end{proof}

\smallskip\noindent
{\it Claim 5}.\quad
    ${\uparrow} b$ is a distributive semilattice.

\begin{proof}
    Suppose $x,y,z\in \up b$ and $x\cap y\subseteq z$.  We must produce $x'\supseteq x$ and $y'\supseteq y$ with $x'\cap y'=z$. We have two cases:
\begin{enumerate}[(a)]
\item Suppose $z\in \mathcal F_1$ or  $z\in \mathcal C_{012}$. Then $u\cup z\in \up b$ for all $u\in \up b$.   Therefore, we may set $x'=x\cup z$ and $y'=y\cup z$.
\item  Suppose $z\in \mathcal F_{01}$ or $z\in \mathcal F_{12}$.  These cases are symmetric under the transposition of $0$ and $2$, so we may assume without loss of generality that $z\in \mathcal F_{01}$. Then $u\cup z\in \up b$ unless $u\in \mathcal F_{12}$.
    If neither $x$ nor $y$ is in $\mathcal F_{12}$, we may set $x'=x\cup z$ and $y'=y\cup z$.  Otherwise, without loss of generality $x\in \mathcal F_{12}$.  In this case, we must have $y\in \mathcal F_1\cup\mathcal F_{01}$ because $2\not\in z$, and therefore $2\not\in x\cap y$.  We set  $x' = \{0,1,2\}\cup(\N_{\geq 3}\setminus y)\cup x\cup z$ and $y'=y\cup z$.
    Then $x'\supseteq x$, $y'\supseteq y$, and $ x'\cap y'$ is the union of the following sets:
\begin{align*}
\{0,1,2\}\cap(y\cup z) &= \{0,1\}\ (\subseteq z),\\
\left(\N_{\geq 3}\setminus y\right)\cap(y\cup z)&=( \N_{\geq 3}\setminus y)\cap z,\\
(x\cup z)\cap(y\cup z)&=z,
\end{align*}
so $x'\cap y'=z$.\qedhere
\end{enumerate}
\end{proof}
\noindent
It is useful to note that $A_{012}=\{\varnothing, \{0\}, \{1\}, \{0,1\}, \{1,2\}\}$ as defined above is closed under all unions \textit{except}
$\{0\}\cup \{1,2\}$ and $\{0,1\}\cup\{1,2\}$.  Thus, a union of two elements of $A$ belongs to $A$ unless one belongs to $\mathcal F_0$ and the other to $\mathcal F_{12}$, or,
one belongs to $\mathcal F_{01}$ and the other to $\mathcal F_{12}$.

\medskip\smallskip\noindent
{\it Claim 6}.\quad
    $A$ is a distributive semilattice.

\begin{proof}
    Suppose $x,y,z\in A$ and $x\cap y\subseteq z$.  We must produce $x'\supseteq x$ and $y'\supseteq y$ with $x'\cap y'=z$.  We have four cases.
 \begin{enumerate}[(i)]
\item Suppose $z\in \mathcal F\cup \mathcal F_1  \cup \mathcal C_{012}$. Then $u\cup z\in A$ for all $u\in A$.   Therefore, we may set $x'=x\cup z$ and $y'=y\cup z$.

\item Suppose  $z\in \mathcal F_0$.  Then $u\cup z\in A$ and we argue as in (i), unless $u\in \mathcal F_{12}$.  If one of $x,y \in \mathcal F_{12}$, we may assume without loss of generality that $x\in \mathcal F_{12}$.  Then $y\in\mathcal F\cup\mathcal F_0$.  In this case, let
 $x' = \{0,1,2\}\cup(\N_{\geq 3}\setminus y)\cup x\cup z$,  and  $y'=y\cup z$, and argue as in Claim 5(b).

\item Suppose $z\in \mathcal F_{01}$. In this case $u\cup z\in A$  and we argue as in (i), unless $u\in \mathcal F_{12}$.
If $x\in \mathcal F_{12}$, then $y\in\mathcal F\cup\mathcal F_0\cup\mathcal F_1\cup\mathcal F_{01}$ and we proceed as in (ii).

\item Suppose $z\in \mathcal F_{12}$. In this case $u\cup z\in A$  and we argue as in (i), unless $u\in \mathcal F_0\cup \mathcal F_{01}$.  We cannot have both $x$ and $y$ in $\mathcal F_0\cup \mathcal F_{01}$ because then $x\cap y$ would contain $0$ and would not be a subset of $z$.  So, suppose that $x\in \mathcal F_0\cup \mathcal F_{01}$.  Then $y\in \mathcal F_{1}\cup \mathcal F_{12}$ and we proceed as in (ii).\qedhere
\end{enumerate}
\end{proof}

\noindent
A topological perspective on this example is presented at the end of \cref{secTopology}. 

\section{Distributive envelopes and subfitness} \label{secEnvelope}

\noindent
In this section
we show that to each join-semilattice $A$ -- distributive or not -- one can associate a distributive lattice $L$ such that $A$ is join-subfit iff $L$ is join-subfit (and analogously for meets); see \cref{SubfitMirroredInDenvelope}. The construction of $L$ from $A$ is functorial and will be studied in a forthcoming paper. If $A$ happens to be weakly distributive in the sense of \cite{balbes_representation_1969,varlet_on_separation_1975}, then $L$ is the distributive envelope of $A$ defined in
\cite{cornish_weakly_1978} (and also studied in \cite{hansoul_priestley_2008,bezhanishvili_priestley_2011}); we won't use this explicitly but leave comments en route to show the connection. Instead we will take $L$ as the sublattice generated by $A$ in its injective hull following \cite{bruns_injective_1970,horn_the_category_1971} in its incarnation for join-semilattices.

We need one notion:
We call a subset $S$ of a join-semilattice $A$ \textbf{admissible} if $a=\bigwedge_AS$ exists and
for every $b\in A$ the set $S\lor b:=\{s\lor b\st s\in S\}$ has infimum $a\lor b$. 
This property plays a key role in \cite{bruns_injective_1970,horn_the_category_1971} and has been considered already by MacNeille in \cite[Def.~3.10]{macneille_partially_1937}.\footnote{Hence, by definition, a join-semilattice $A$ is weakly distributive if every finite subset that has an infimum in $A$ is admissible.}

\goodbreak
\begin{proposition}\label{SubfitExt}
Let $A\subseteq B$ be an extension of bounded join-semilattices such that:
\begin{enumerate}[{\rm (a)}]

\item\label{SubfitExtT}
for every $b\in B$ there are $a_1,\ldots,a_n\in A$ such that
for all $a\in A$ we have
\[
a\lor b=\bigwedge\nolimits_B \{ a\lor a_1,\ldots,a\lor a_n\},\footnote{If $B$ is weakly distributive, this simply says that $A$ is finitely meet-dense in $B$, i.e. every element of $B$ is a finite meet of elements from $A$.}
\]

\item\label{SubfitExtDist}
for every finite admissible subset $F$ of $A$, $\bigwedge_A F = \bigwedge_B F$.
\end{enumerate}
Then $A$ is \jsub\ if and only if $B$ is \jsub.
\end{proposition}
\begin{proof}
First suppose $B$ is \jsub.
Take $u,v\in A$ with $u\nleq v$. Since $B$ is \jsub, there is $b\in B$ such that $v\lor b<1$ and $u\lor b=1$.
Take
$a_1,\ldots,a_n\in A$ for $b$ as in \ref{SubfitExtT}.
Then $\bigwedge_B \{ v\lor a_1, \ldots, v\lor a_n\} = v\lor b<1$, so $v\lor a_i<1$ for some $i$. Since $b\leq a_i$ (set $a=0$ in \ref{SubfitExtT}) we have $u\lor a_i=1$, yielding that $A$ is \jsub.

Conversely, suppose that $A$ is \jsub. Let $u,v\in B$ with $u\nleq v$.
We need to find $w\in B$ such that  $v\lor w<1$ and $u\lor w=1$.

By \ref{SubfitExtT} there are $\tilde a_1,\ldots,\tilde a_k\in A$ such that
$v=\bigwedge_B \{ \tilde a_1, \ldots, \tilde a_k \}$.
Since $u\nleq v$, there is $i$ with $u\nleq \tilde a_i$.
We may then replace $v$ by $\tilde a_i$ and assume that $v\in A$.
Using \ref{SubfitExtT} again, there are $a_1,\ldots,a_n\in A$ such that

\begin{enumerate}[leftmargin=5ex]
\item[($\dagger$)]
for all $c\in A$ we have
$c\lor u = \bigwedge_B \{c\lor a_1,\ldots,c\lor a_n\}$;
in particular, $u=\bigwedge_B\{a_1,\ldots,a_n\}$.
\end{enumerate}

\smallskip\noindent
{\em Claim}.
    There is $c\in A$ such that
$c\lor v$ is \textit{not} the infimum of $\{c\lor a_1\lor v,\ldots,c\lor a_n\lor v\}$
in $A$.

\begin{proof}[Proof of Claim]
    Suppose for all $c\in A$ we have
$c\lor v = \bigwedge_A \{c\lor a_1\lor v,\ldots,c\lor a_n\lor v\}$.
Then $v=\bigwedge\nolimits_A\{a_1\lor v,\ldots,a_n\lor v\}$ and
$\{a_1\lor v,\ldots,a_n\lor v\}$ is an admissible subset of $A$. Therefore,
by \ref{SubfitExtDist},
$v=\bigwedge\nolimits_B\{a_1\lor v,\ldots,a_n\lor v\}$.
On the other hand, by $(\dagger)$,
\[
\bigwedge\nolimits_B\{a_1\lor v,\ldots,a_n\lor v\}=v\lor \bigwedge\nolimits_B\{a_1,\ldots,a_n\}=v\lor u,
\]
which implies $u\leq v$, a contradiction.
\end{proof}

\noindent
Now take $c\in A$ as in the claim, i.e.~there is $w_0\in A$ such that $w_0\leq c\lor a_i\lor v$ for all $i$ but $w_0\nleq c\lor v$.
Because $A$ is \jsub, there is $w_1\in A$ with
$c\lor v\lor w_1<1$ and $w_0\lor w_1=1$.
Then $w:=c\lor v\lor w_1$ satisfies $v\lor w<1$ and
\begin{align*}
    u\lor w&=u\lor c\lor v\lor w_1\cr
    &=\bigwedge\nolimits_B\{a_1,\ldots,a_n\}\lor c\lor v\lor w_1\cr
    &\stackrel{(\dagger)}{=}\bigwedge\nolimits_B\{a_1\lor c\lor v,\ldots,a_n\lor c\lor v\}\lor w_1\cr
    &\geq w_0\lor w_1=1,
\end{align*}
showing that $B$ is \jsub.
\end{proof}

\noindent
We now invoke \cite{bruns_injective_1970} in its order-dual form or rather the following consequence:
Every join-semilattice $A$ has an injective hull $E$ in the category of join-semilattices with semilattice homomorphisms. Furthermore, the following properties hold:
\begin{enumerate}[leftmargin=5ex]
\item $E$ is a distributive lattice.\footnote{In fact it is a co-frame, i.e. the order-dual of a frame.}
\item For each admissible subset $S$ of $A$, the infimum of $S$ in $A$ is the infimum of $S$ in $E$.
\end{enumerate}
It follows that the sublattice of $E$ generated by $A$ has properties (a) and (b) of
\cref{SubfitExt}. We thus obtain the following consequence.

%

\begin{theorem}\label{SubfitMirroredInDenvelope}
Let $A$ be a bounded
join-semilattice and let $E$ be the injective hull of $A$.
We set $L$ to be the bounded sublattice of $E$ generated by $A$.
Then $L$ is a bounded distributive lattice and $A$ is \jsub\ if and only if $L$ is \jsub.\qed
\end{theorem}

\begin{remark}
As mentioned earlier, the construction of $L$ from $A$ in \cref{SubfitMirroredInDenvelope} is functorial, see for example
\cite[Cor.~3.13]{gehrke_distributive_2014}. Hence,
compared to our results in \cref{secLattice,DSLatCounterExample}, \cref{SubfitMirroredInDenvelope} is somewhat unexpected in that it asserts that
\jsub ness is transferable between $A$ and $L$,
yet in \cref{SubfitIsJoinStableInDLat} we have seen that \jsub\ elements of a distributive lattice do form an ideal,
whereas in \cref{DSLatCounterExample} we have seen that this is no longer the case in a distributive join-semilattice.
The crux of the matter is that the construction in \cref{SubfitMirroredInDenvelope} does not commute with taking principal downsets.
The situation with \msub ness is analogous.
\end{remark}

\bigskip

\section{Topological perspective} \label{secTopology}

\noindent
Join-semilattices occur naturally in non-Hausdorff topology. For a topological space $X$ (which always means T$_0$ here), the set
\[
\qcop(X)=\{U\subseteq X\st U\text{ is compact open}\}
\]
defines a join-semilattice, where the join is given by union. The space $X$ is
\textbf{compactly based} provided $\qcop(X)$ is a basis for $X$.
By the Gr\"atzer-Stone representation theorem (see  \cite[Thm. 191, p. 171]{gratzer_lattice_2011}), 
every bounded distributive join-semilattice  $A$ is isomorphic to $\qcop(X)$ for a unique
compactly based, compact, and sober space $X$ and each such space has a distributive $\qcop(X)$.
In this correspondence, by definition, $\qcop(X)$ is a lattice if and only if $X$ is a
\textbf{spectral space} in the sense of \cite{hochster_prime_1969}. The matter is outlined in more detailed in
\cite[3.7.2]{dickmann_spectral_2019}.

We may therefore describe the results of this paper in terms of the
join-semilattice $\qcop(X)=(\qcop(X),\cup )$ for a compact and compactly based space.
We point out that this section can be phrased in different languages, because compact, compactly based spaces come in many different flavors.
For example, as generalized Priestley spaces (see \cite{hansoul_priestley_2008,bezhanishvili_priestley_2011}), as spectra of predomains with proximity
(see \cite[Rem. 3.16]{bice_gratzer_2021}),
and as (point spaces of) algebraic frames (see, e.g., \cite[3.7.2]{dickmann_spectral_2019}).

We give abridged proofs with reference to \cite{dickmann_spectral_2019}, thereby providing access to the topological point of view. Notice that the results themselves are first-order and constructive. Only when they are linked to their lattice-theoretic counterparts, a second-order argument and the axiom of choice are needed.

\begin{remark}\label{TopGeneral}Let $X$ be any T$_0$-space.
\begin{enumerate}[(i)]
\item\label{TopGeneralDefnCP}
We write $\cp X$ for the subspace of closed points of $X$, i.e.~those points $x\in X$ for which $\{x\}$ is closed.
\item\label{TopGeneralCompact}
$X$ is compact if and only if $\cp X$ is compact and every nonempty closed subset of $X$ contains a closed point, cf.
\cite[4.1.2]{dickmann_spectral_2019}.
\item\label{TopGeneralCompactPatch}
The \textbf{patch-topology} of $X$ is defined to be the topology of $X$ generated by the sets $U\setminus V$, where $U,V\in\qcop(X)$.
Observe that if $X$ is spectral, this is the patch-topology in the sense of \cite{hochster_prime_1969,dickmann_spectral_2019}.
\end{enumerate}
\end{remark}

\noindent
Here is a characterization of \jsub ness of $\qcop(X)$ when $X$ is compact and compactly based.\looseness=-1
\begin{proposition}\label{CharacterizeQCOPsubfit}
Let $X$ be a compact and compactly based space.
Then $\qcop(X)$ is \jsub\ if and only if $\cp X$ is patch-dense.
\end{proposition}
\begin{proof}
First suppose $\cp X$ is dense for the patch-topology of $X$.
Take $U,V\in\qcop(X)$ with $U\nsubseteq V$ and take a closed point $p\in U\setminus V$.
Since $\qcop(X)$ is a basis and $p$ is a closed point,
\(
(X\setminus U)\cup V\subseteq X\setminus \{p\}=\bigcup\{ W : p\notin W\in \qcop(X) \}.
\)
Because $(X\setminus U)\cup V$ is \qc\ and $\qcop(X)$ is closed under finite unions, there is $W\in \qcop(X)$ with $(X\setminus U)\cup V\subseteq W\not\ni p$.
Thus, $p\notin W\cup V$ and $U\cup W=X$, showing that $\qcop(X)$ is \jsub.

Conversely, assume that $\qcop(X)$ is \jsub.
Since $\qcop(X)$ is a basis, the sets of the form $U\setminus V$ with $U,V\in\qcop(X)$ form a basis for the patch-topology of $X$.
So we only need to show that each nonempty such set has a closed point.
As $\qcop(X)$ is \jsub, there is $W\in\qcop(X)$ with $U\cup W=X$ and $V\cup W\neq X$.
This implies $\varnothing\neq X\setminus (V\cup W)\subseteq U\setminus V$.
Now, $X\setminus (V\cup W)$ is a nonempty closed set of the compact space $X$, hence it contains a closed point of $X$ by
\cref{TopGeneral}\ref{TopGeneralCompact}.
\end{proof}

\goodbreak
\begin{corollary}
For every spectral space $X$, the lattice
$\qcop(X)$ is \msub\ if and only if each $U\in \qcop(X)$ is regular open.
\end{corollary}
\begin{proof}
For a spectral space $X$ the lattice $\qcop(X)$ is also a basis of closed sets of another spectral space, called the \textbf{inverse space} $X_\inv$ of $X$, see, e.g., \cite[sec.~1.4]{dickmann_spectral_2019}.
The space $X_\inv$ has the same patch-topology as $X$ and $\qcop(X_\inv)=\{A\subseteq X\st A\text{ closed in }X\text{ and patch-open}\}$ is the order-dual lattice of $\qcop(X)$. Furthermore, the closed points of $X_\inv$ are the minimal points of $X$ in the setup of Corollary 4.4.20 in \cite{dickmann_spectral_2019}. This corollary then says that the set of minimal points of $X$ is patch-dense iff
every $U\in \qcop(X)$ is regular open, i.e. $U$ is the interior for the topology of $X$ of the closure of $U$ in $X$. Consequently, \cref{CharacterizeQCOPsubfit} applied to $X_\inv$ yields the result.\looseness=-1
\end{proof}

\smallskip\noindent
\textbf{Join-subfitness of compact open sets of prime spectra of rings.}
A \textbf{Yosida frame} is an algebraic frame whose semilattice of compact elements is ideally subfit, see \cite[sec. 3.5]{delzell_conjunctive_2021}.
Here we are specifically interested in compact and coherent frames, because these are precisely those that are isomorphic to the topology
of some spectral space.  If $X$ is a spectral space then the lattice $\qcop(X)$ is \jsub\ if and only if $\CO(X)$ is a Yosida frame
(see \cite[Prop.~4.2]{martinez_yosida_2006} and \cite[Prop.~3.20]{delzell_conjunctive_2021} for a more general statement).

In \cite[sec. 5 and 6]{martinez_yosida_2006} one can find a host of examples of $\ell$-groups whose frame of convex $\ell$-subgroups is a Yosida frame and examples of B\'ezout domains $R$
for which $\qcop(\Spec(R))$ is join-subfit. We add two more examples.

Firstly, let $C(T)$ be the ring of continuous real valued functions on a completely regular space $T$.
Then the lattice $\qcop(X)$ with $X=\Spec C(T)$ is \jsub\ if and only if $T$ is a P-space, meaning that $C(T)$ is von Neumann regular.
The reason is that $\cp X$ is the set of maximal ideals of $C(T)$ and the patch-closure $Z$ of the maximal ideals is the set of prime z-ideals of $C(T)$
\footnote{The frame $\CO(Z)$ is the archetypical example of a Yosida frame. See the introduction of \cite{martinez_yosida_2006}
and also \cite[13.5.16]{dickmann_spectral_2019}.}; hence the only way $\cp X$ can be patch-dense is when
$X=Z$. But this is only possible for P-spaces, see \cite[4J]{gillman_rings_1960}.

The situation changes when we restrict to semi-algebraic functions: Let $R$ be the ring of continuous semi-algebraic functions
$f:\Rn^n\lra \Rn$, i.e.~the graph of $f$ is first-order definable in the field $\Rn$. We show that $\qcop(\Spec(R))$ is \jsub\ by verifying
\cref{CharacterizeQCOPsubfit}, i.e., we show that every nonempty patch-open subset $C$ of $\Spec(R)$ contains a maximal ideal of $R$.
Using \cite[12.1.10(iv)]{dickmann_spectral_2019}, we may assume that $C$ is of the form $\{\Dp\in\Spec(R)\st f\in \Dp,\ g\notin \Dp\}$ for some $f,g\in R$, because in $R$,
the radical of any finitely generated ideal $(f_1,\ldots,f_k)$ is equal to the radical ideal generated by $f_1^2+\ldots+f_k^2$ (the function $\frac{f_i^3}{f_1^2+\ldots+f_k^2}$ has a continuous extension through $0$).
Since $C$ is nonempty, no power of $g$ is divisible by $f$ in $R$. By the \Lpolish ojasiewicz inequality of Real Algebraic Geometry, see \cite[Thm.~2.6.6]{bochnak_real_algebraic_1998}, we cannot have
$f^{-1}(0)\subseteq g^{-1}(0)$, hence there is some $a\in\Rn^n$ with $f(a)=0$ and $g(a)\neq 0$. Let $e:R\lra \Rn$ be the evaluation map at $a$.
Then $e$ is a ring homomorphism with image $\Rn$ and so its kernel $\Dm$ is a maximal ideal of $R$.
By the choice of $a$ we get $\Dm\in C$.

\goodbreak
\smallskip\noindent
\textbf{Topological perspective on \cref{SubfitIsJoinStableInDLat}}\qquad Let $X$ be a spectral space.
Using  \cref{CharacterizeQCOPsubfit} and Stone duality for distributive lattices, \cref{SubfitIsJoinStableInDLat}
is equivalent to saying that for all $U,V\in\qcop(X)$, if $\cp U$ is patch-dense in $U$ and $\cp V$ is patch-dense in $V$, then $\cp{U\cup V}$ is patch-dense in $U\cup V$.

We give a direct proof of this fact,
thus providing a topological perspective on \cref{SubfitIsJoinStableInDLat}.
The key (and only) property used in this proof that is not available for compactly based spaces in general is the following
consequence of $\qcop(X)$ being stable under finite intersections:
\begin{enumerate}
\item[$(*)$] If $O_1,O_2\in\qcop(X)$ with $O_1\subseteq O_2$, then for every patch open subset $C$ of the space $O_2$,
the set $C\cap O_1$ is empty or has nonempty interior for the patch topology of $O_1$.\footnote{In fact, we even know that the patch topology of $O_1$ is the subspace topology induced on $O_1$ by the patch topology of $O_2$. The weaker assumption $(*)$ is used here because the proof also goes through for compact, compactly based spaces only satisfying $(*)$.}
\end{enumerate}

\noindent
We need to show that
$\cp{U\cup V}$ is patch-dense in $U\cup V$.
Take $C\subseteq U\cup V$ nonempty and patch-open. By symmetry we may assume that $C\cap V\neq \varnothing$.
By $(*)$ and the assumption on $V$, there is some $v\in \cp V\cap C$.
If $v\notin U$, then since $U$ is open, $v$ is a closed point of $U\cup V$ and we are done.
Hence, we may assume that $v\in U$.

From $v\in \cp V$ we know $\overline{\{v\}}\cap V=\{v\}$ and so $\cp V\setminus U \subseteq X\setminus \overline{\{v\}}$.
Since $\qcop(X)$ is a basis and $\cp V\setminus U$ is compact (use \cref{TopGeneral}\ref{TopGeneralCompact}), there is some $W\in \qcop(X)$
with $\cp V\setminus U \subseteq W\subseteq X\setminus \overline{\{v\}}$. Consequently, $v\in U\cap C\setminus W$, showing that
$U\cap C\setminus W$ is nonempty. By $(*)$ and the assumption on $U$, there is some $u\in \cp U\cap C\setminus W$.
If $u\notin V$, then since $V$ is open, $u\in \cp{U\cup V}$ and we are done.
Hence, we may assume that $u\in V$. If $u\in \cp V$, then we are done. Otherwise
there is $w\in \cp V$ such that $w\in \overline{\{u\}}$ and $w\neq u$.
But then $w\notin U$ because $u\in \cp U$. Thus, $w\in \cp V\setminus U\subseteq W$
and so $u\in W$, a contradiction.
\qed

\goodbreak
\smallskip\noindent
\hypertarget{CompBasedCounterExample}{\textbf{The counterexample of \cref{DSLatCounterExample} in topological form}}\qquad
We present the compactly based space
$X$ for which $\qcop(X)$ is the order-dual of the distributive meet-semilattice from  \cref{DSLatCounterExample} in the Gr\"atzer-Stone representation.
Hence the \jsub\ elements of $\qcop(X)$ do not form an ideal.
We work in an auxiliary spectral space $Y$, shown below.

\medskip
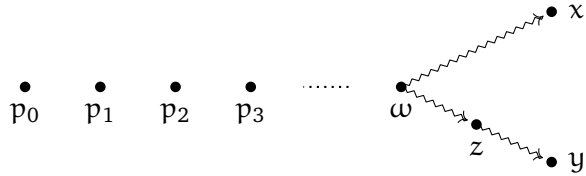
\tikzstyle{squiggly} = [line join=round,decorate, decoration={
    zigzag,
    segment length=4,
    amplitude=.9,post=lineto,
    post length=2pt
}]
\begin{center}
\begin{tikzpicture}[scale=1]
    \coordinate (omega) at (5,0);
    \foreach \r in {0,...,3}
      {
       \coordinate (p) at (\r,0);
       \coordinate (pnext) at (\r+1,0);
       \fill[black] (p) circle (2pt) node[below=0.5ex] {$p_\r$};
       \ifthenelse{\r = 3}{}{
       }
      }
  \draw[-,dotted,thick,shorten <=1.3cm,shorten >=1.3cm] (p) -- (omega);
  \fill[black] (omega) circle (2pt) node[below=0.5ex] {$\omega$};
  \coordinate (x) at (4,2);
  \coordinate (z) at (5.5,1);
  \coordinate (y) at (6,2);
  \coordinate (x) at (7,1);
  \coordinate (z) at (6,-0.5);
  \coordinate (y) at (7,-1);
  \draw[->,squiggly,shorten <=0.05cm,shorten >=0.15cm] (omega) -- (x);
  \fill[black] (x) circle (2pt) node[right=0.5ex] {$x$};
  \draw[->,squiggly,shorten <=0.05cm,shorten >=0.15cm] (omega) -- (z);
  \fill[black] (z) circle (2pt) node[below=0.5ex] {$z$};
  \draw[->,squiggly,shorten <=0.05cm,shorten >=0.15cm] (z) -- (y);
  \fill[black] (y) circle (2pt) node[right=0.5ex] {$y$};
\end{tikzpicture}
\end{center}

\medskip\noindent
Here $\omega$ is the only non-isolated point of the patch topology of $Y$ and the immediate specializations are those indicated by a squiggly
arrow\footnote{Hence $u\rightsquigarrow v$ implies that $v\in \overline{\{u\}}$.} 
(hence all points of $Y$ except $\omega$ and $z$ are closed points).
Let $X$ be the subspace $Y\setminus \{\omega\}$ of $Y$ and let $P=\{p_0,p_1,p_2,\ldots\}$.
One checks that for $U\in\qcop(Y)$ the set $U\cap X$ is \qc\ provided $U\neq P_0\cup\{\omega\}$ for all cofinite subsets $P_0$ of $P$.
It follows that $X$ is compactly based. However $\qcop(X)$ is not \jsub\ by \cref{CharacterizeQCOPsubfit} since $\cp X$ is not patch-dense,
witnessed by the identity $\{z\}=(P\cup \{z\})\setminus (P\cup \{x\})$.

Finally, $X=V\cup W$ with $V=P\cup \{x\}\in \qcop(X)$ and $W=P\cup \{y,z\}\in \qcop(X)$, but $\qcop(V)$ and $\qcop(W)$  are \jsub\ as one checks
using \cref{CharacterizeQCOPsubfit} and the characterization of compactness above.

\bibliographystyle{alpha}
\bibliography{\jobname}

\end{document}